\documentclass[12pt]{article}
\usepackage[T1]{fontenc}
\usepackage{amssymb,bbm,amsfonts,mathrsfs,graphicx,amsmath,amsthm,xypic}
\usepackage{times,geometry}
\usepackage{mathtools,microtype,booktabs,array,tabularx}
\usepackage{xcolor,needspace,float}
\usepackage{hyperref}

\newtheorem{Th}{\scshape Theorem}[section]
\newtheorem{Lem}[Th]{\scshape Lemma}
\newtheorem{Cor}[Th]{\scshape Corollary}

\newtheorem{Prop}[Th]{\scshape Proposition}

\numberwithin{equation}{section}

\newcommand{\sig}{\sigma}
\newcommand{\cl}{\mathcal{G}_3}
\newcommand{\FQ}{\mathrm{FQ}_7}
\newcommand{\Odd}{O_4}
\DeclareMathOperator{\tr}{tr}
\DeclareMathOperator{\rank}{rank}

\hypersetup{
    pdftitle={Spectral bipartiteness in generalized odd graphs of diameter three},
    pdfsubject={Exact spectral extrema with a complete computational certificate},
    pdfkeywords={distance-regular graphs, odd girth, spectral bipartiteness, exact computation},
    colorlinks=true,
    linkcolor=blue!45!black,
    citecolor=blue!45!black,
    urlcolor=blue!45!black
}

\makeatletter
\renewcommand{\@maketitle}{
    \newpage
    \null
    \begin{center}
        \let\footnote\thanks
        {\LARGE \@title\par}
        \vskip 0.5em
        {\large
            \lineskip 0.5em
            \begin{tabular}[t]{c}
                \@author
            \end{tabular}\par
        }
        \vskip 0.5em
        {\large \@date\par}
    \end{center}
    \par
    \vskip 0.5em
}
\makeatother

\begin{document}
    \raggedbottom

    \title{Spectral bipartiteness in generalized odd graphs of diameter three}
    \author{
        Qi Zhou \thanks{Corresponding author: imzhouqi@126.com}
        \\ \small \it School of Finance and Mathematics, Huainan Normal University, Huainan, China
        }
    \date{}
    \maketitle

    \noindent {\bf Abstract:} For a graph $G$ of order $n$, put $\sigma(G)=(\lambda_1(G)+\lambda_n(G))/n$. We determine the first three largest values of this invariant among nonbipartite distance-regular graphs of diameter three and odd girth at least seven. The unique maximizer is the folded $7$-cube, with value $1/32$; the unique second maximizer is the Odd graph $O_4$, with value $1/35$; and the unique third maximizer is $C_7$, with value $2(1-\cos(\pi/7))/7$. More precisely, every other graph in the class satisfies $\sigma(G)<1/36$. This answers Problem~11 of Abiad, Taranchuk and van Veluw in \emph{Electronic Journal of Combinatorics} 33(2) (2026), P2.31. The proof combines established local multiplicity and odd-moment bounds: the condition $\sigma(G)\geq1/36$ forces the valency to be at most $182$. An exhaustive certificate using only integer and rational arithmetic then leaves three intersection arrays. The complete certificate is publicly available, and neither a classification of generalized odd graphs nor the $Q$-polynomial property is assumed. The odd-girth theorem gives the same extremal conclusions for connected $\{C_3,C_5\}$-free graphs with at most four distinct adjacency eigenvalues, without assuming regularity.



    \section{Introduction and main results}
    All graphs in this paper are finite, undirected and simple. Write $\lambda_1(G)\geq\cdots\geq\lambda_n(G)$ for the adjacency eigenvalues of a graph $G$ of order $n$, and define
    \[
        \sig(G)=\frac{\lambda_1(G)+\lambda_n(G)}{n}.
    \]
    For a connected graph, $\sig(G)=0$ if and only if $G$ is bipartite. For a $k$-regular graph, $\lambda_1(G)=k$ and $\lambda_1(G)+\lambda_n(G)$ is also the least signless Laplacian eigenvalue. The normalization is essential when graphs of different orders are compared.

    These spectral questions are closely related to the classical problem of making a graph bipartite by deleting edges. Let $\tau_{\mathrm b}(G)$ be the minimum number of edges required, equivalently $|E(G)|-\operatorname{MaxCut}(G)$, where $\operatorname{MaxCut}(G)$ is the maximum cut size. Erd\H{o}s, Faudree, Pach and Spencer~\cite{EFPS1988} established bounds in terms of the order and size, including $\tau_{\mathrm b}(G)\leq(1/18-\varepsilon+o(1))n^2$ for triangle-free graphs, for some absolute $\varepsilon>0$. Their arguments use large induced bipartite subgraphs. Erd\H{o}s, Gy\H{o}ri and Simonovits~\cite{EGS1992} showed that a triangle-free $n$-vertex graph with at least $n^2/5$ edges is no harder to make bipartite than a suitable blow-up of $C_5$ on $n$ vertices with at least as many edges, and proved a related stability theorem.

    Balogh, Clemen and Lidick\'{y}~\cite{BCL2021} used flag algebras and structural arguments to prove $\tau_{\mathrm b}(G)\leq n^2/23.5$ for sufficiently large triangle-free graphs.

    Fallat and Fan~\cite{FF2012} bounded the least signless Laplacian eigenvalue $q_n(G)$ by the minimum vertex and edge deletions required for bipartiteness. De Lima, Nikiforov and Oliveira~\cite{DLNO2016} studied its maximization under clique exclusion.

    For triangle-free graphs, Balogh, Clemen, Lidick\'{y}, Norin and Volec~\cite{BCLNV2023} obtained $q_n(G)/n\leq15/94$ using Rayleigh quotients, induced-path counts and flag algebras; for regular graphs this bounds $\sig(G)$. Csikv\'{a}ri~\cite{Csi2022} identified the Higman--Sims graph as the unique maximizer of $\sig$ among triangle-free strongly regular graphs, with value $0.14$. This diameter-two result is a direct antecedent of our diameter-three classification, in which pentagons are also forbidden.

    For a $k$-regular graph, the elementary Rayleigh-quotient bound~\cite{DLNO2016} gives a direct connection between the deletion parameter and the present spectral invariant:
    \[
        \frac{\sig(G)}4=\frac{k+\lambda_n(G)}{4n}
        \leq\frac{\tau_{\mathrm b}(G)}{n^2}.
    \]
    Indeed, a sign vector $s\in\{-1,1\}^n$ representing a maximum cut satisfies $s^{\mathsf T}(kI+A)s=4\tau_{\mathrm b}(G)$ and $s^{\mathsf T}s=n$. Thus the present theorem identifies the graphs for which this normalized spectral lower bound is largest within the distance-regular class of diameter three and odd girth seven, and determines its next two values. This gives an exact spectral counterpart to the classical study of how forbidden odd cycles constrain the distance from bipartiteness.

    Abiad, Taranchuk and van Veluw~\cite{ATV2026} proved $\sig(G)<0.0396$ for all $\{C_3,C_5\}$-free graphs, using weighted eigenvalue interlacing. Their Problem~11, on page~11 of the published article, asks whether the folded $7$-cube maximizes $\sig$ among $\{C_3,C_5\}$-free distance-regular graphs of diameter three. Yip~\cite{Yip2026} subsequently obtained an $O(g^{-3}(\log g)^3)$ bound as the odd-girth lower bound $g$ tends to infinity, using approximation theory. That asymptotic result does not identify the extremizer in the fixed diameter-three problem.

    A nonbipartite distance-regular graph of diameter $D$ and odd girth $2D+1$ is called a \emph{generalized odd graph}. Let $\cl$ denote the class of nonbipartite distance-regular graphs of diameter three having no triangles or pentagons. These graphs have odd girth exactly seven; see Section~\ref{sec:prelim}. Let $\FQ$ denote the folded $7$-cube, and let $\Odd=KG(7,3)$ denote the Odd graph whose vertices are the $3$-subsets of a $7$-set, with disjointness as adjacency.

    The contribution is an extremal classification obtained by combining established spectral tools with an explicit finite reduction. The local multiplicity bound comes from Terwilliger\textquotesingle s tree bound~\cite{Ter1982} and the neighborhood Gram-matrix argument of Juri\v{s}i\'{c}, Koolen and Miklavi\v{c}~\cite{JKM2005}; the odd-moment inequality is the regular-graph specialization of Abiad, Taranchuk and van Veluw~\cite{ATV2026}. Their combination with the normalized objective gives the cutoff $k\leq182$, after which exact arithmetic excludes all competing arrays. Finite valency reductions also occur in Qiao and Koolen~\cite{QK2018}. The present exclusion assumes neither a classification of generalized odd graphs nor the $Q$-polynomial property considered by Lang and Terwilliger~\cite{LT2007}.

    Bang and Koolen~\cite{BK2017} studied distance-regular graphs of diameter three with eigenvalue $-1$, obtaining finiteness and classification results under additional parameter restrictions. This eigenvalue occurs in $\Odd$ and $\FQ$, but not in $C_7$. Their work provides further structural context for the equality cases; our exclusion tests do not assume the presence of $-1$.

    \begin{Th}\label{thm:threshold}
        For $G\in\cl$,
        \[
            \sig(G)\geq\frac1{36}
            \quad\Longleftrightarrow\quad
            G\cong C_7,\ \Odd,\ \text{or }\FQ.
        \]
    \end{Th}

    \begin{Cor}\label{cor:extrema}
        For $G\in\cl$, the following statements hold.
        \begin{enumerate}
            \item $\sig(G)\leq1/32$, with equality if and only if $G\cong\FQ$.
            \item If $G\not\cong\FQ$, then $\sig(G)\leq1/35$, with equality if and only if $G\cong\Odd$.
            \item If $G$ is isomorphic to neither $\FQ$ nor $\Odd$, then
            \[
                \sig(G)\leq\frac{2(1-\cos(\pi/7))}{7},
            \]
            with equality if and only if $G\cong C_7$.
        \end{enumerate}
        Every graph outside these three isomorphism classes satisfies $\sig(G)<1/36$.
    \end{Cor}

    The three values are
    \[
        0.03125,\qquad 0.028571428571\ldots,\qquad
        0.028294609170\ldots.
    \]
    The threshold $1/36$ is a convenient rational separator below $\sig(C_7)$; the theorem does not determine the fourth-largest value. The extremal value is separated from the rest of $\cl$ by $1/32-1/35=3/1120$. Bipartite graphs have $\sig=0$, so the first conclusion also answers the published problem \cite{ATV2026} in its full stated class.

    \section{Graph parameters and bounds}\label{sec:tools}
    \subsection{Graph parameters and eigenvalues}\label{sec:prelim}
    For a vertex $v$ of a connected graph, let $\Gamma_i(v)$ be the set of vertices at distance $i$ from $v$. In a distance-regular graph, a vertex $w\in\Gamma_i(v)$ has $c_i$, $a_i$, and $b_i$ neighbors in $\Gamma_{i-1}(v)$, $\Gamma_i(v)$, and $\Gamma_{i+1}(v)$, respectively, where these numbers depend only on $i$. The valency is $k=b_0$, and $a_i+b_i+c_i=k$. We use the standard theory of intersection matrices and primitive idempotents; see~\cite{BCN1989}.

    For $G\in\cl$, the absence of triangles and pentagons gives $a_1=a_2=0$. Indeed, an edge within a distance layer of index at most two creates an odd closed walk of length at most five and therefore an odd cycle of length at most five. Also $a_3>0$: if every $a_i$ vanished, the even and odd distance layers would give a bipartition. Consequently, with $\mu=c_2$ and $c=c_3$, the intersection array is
    \begin{equation}\label{eq:array}
        \{k,k-1,k-\mu;1,\mu,c\},\qquad
        k\geq2,\quad 1\leq\mu\leq c<k.
    \end{equation}
    Here we used the standard monotonicity $c_2\leq c_3$. An edge in the third distance layer produces an odd closed walk of length seven. Since shorter odd cycles are absent, the odd girth is seven.

    Write $k_i=|\Gamma_i(v)|$. Edge counting between consecutive layers gives
    \begin{equation}\label{eq:layers}
        k_0=1,\quad k_1=k,\quad
        k_2=\frac{k(k-1)}{\mu},\quad
        k_3=\frac{k(k-1)(k-\mu)}{\mu c},\quad
        n=1+k+k_2+k_3.
    \end{equation}
    In particular, $k_2$ and $k_3$ are integers. The intersection matrix is
    \begin{equation}\label{eq:L}
        L=\begin{pmatrix}
            0&k&0&0\\
            1&0&k-1&0\\
            0&\mu&0&k-\mu\\
            0&0&c&k-c
        \end{pmatrix}.
    \end{equation}
    It is similar, via $\operatorname{diag}(\sqrt{k_0},\ldots,\sqrt{k_3})$, to a symmetric tridiagonal matrix $B$ with positive off-diagonal entries. Its four simple eigenvalues are precisely the distinct eigenvalues of $G$. The row sums of $L$ are $k$, and $k-c>0$, so all its other eigenvalues lie strictly between $-k$ and $k$.

    The following established bound is a consequence of Terwilliger\textquotesingle s tree bound~\cite{Ter1982}. For $k\geq3$, it and the neighborhood Gram-matrix argument appear in Juri\v{s}i\'{c}, Koolen and Miklavi\v{c}~\cite{JKM2005}, specifically Lemma~7(i)--(ii) on pages~6--7 of their 2004 preprint. We retain the short proof for completeness; it also covers $k=2$ under the stated nonzero-eigenvalue hypothesis.

    \begin{Lem}\label{lem:rank}
        Let $G$ be a triangle-free distance-regular graph of valency $k\geq2$. If $\theta$ is an eigenvalue with $0<|\theta|<k$, then its multiplicity $m(\theta)$ satisfies $m(\theta)\geq k$.
    \end{Lem}
    \begin{proof}
        Let $E$ be the orthogonal projection onto the $\theta$-eigenspace. The distance-regularity of $G$ gives
        \[
            E_{vw}=\frac{m(\theta)}n u_i
            \quad\text{when }\operatorname{dist}(v,w)=i,
            \qquad u_0=1,\quad u_1=\frac\theta k.
        \]
        The standard-sequence recurrence at layer one, where $a_1=0$, gives
        \[
            u_2=\frac{\theta^2-k}{k(k-1)}.
        \]
        The $k$ neighbors of a fixed vertex are mutually at distance two. The corresponding principal submatrix of $E$ is therefore
        \[
            \frac{m(\theta)}n\bigl((1-u_2)I_k+u_2J_k\bigr).
        \]
        Its eigenvalues, after omitting the positive factor $m(\theta)/n$, are
        \[
            1-u_2=\frac{k^2-\theta^2}{k(k-1)}>0
            \quad(k-1\text{ times}),\qquad
            1+(k-1)u_2=\frac{\theta^2}{k}>0.
        \]
        Thus this principal submatrix has rank $k$, while $\rank(E)=m(\theta)$.
    \end{proof}

    \subsection{Bounds for the search}\label{sec:reduction}
    In this subsection the diameter need not be three. For a connected nonbipartite $k$-regular graph, put
    \[
        b=-\lambda_n(G),\qquad x=\frac bk\in(0,1),\qquad
        t=\frac nk.
    \]

    The first inequality below is the specialization of~\cite[Lemma~6]{ATV2026} to a regular graph. Indeed, substituting $\lambda_1=k$ and $\lambda_n=-b$ in that result gives
    \[
        k\leq\frac{b^{2\ell-1}n}{k^{2\ell-1}+b^{2\ell-1}}
        \quad\Longleftrightarrow\quad x^{2\ell-1}(t-1)\geq1.
    \]
    We include its moment proof to make the combination with the multiplicity constraint explicit.

    \begin{Lem}\label{lem:moments}
        Let $G$ be a nonbipartite triangle-free distance-regular graph of valency $k\geq2$. If its odd girth is at least $2\ell+3$, where $\ell\geq1$, then
        \begin{equation}\label{eq:basicbounds}
            x^{2\ell-1}(t-1)\geq1,\qquad
            k\leq\frac{t-1}{x^2}.
        \end{equation}
        In particular,
        \begin{equation}\label{eq:ordergrowth}
            n\geq k+k^{4\ell/(2\ell+1)}.
        \end{equation}
    \end{Lem}
    \begin{proof}
        The absence of odd closed walks of length $2\ell+1$ gives $\tr(A^{2\ell+1})=0$. Thus
        \[
            k^{2\ell+1}
            \leq\sum_{\lambda_i<0}(-\lambda_i)^{2\ell+1}
            \leq b^{2\ell-1}\sum_{\lambda_i<0}\lambda_i^2
            \leq b^{2\ell-1}(nk-k^2).
        \]
        Dividing by $k^{2\ell+1}$ proves the first inequality in \eqref{eq:basicbounds}. If $m$ is the multiplicity of $-b$, Lemma~\ref{lem:rank} and $\tr(A^2)=nk$ give
        \[
            k b^2\leq m b^2\leq nk-k^2,
        \]
        which proves the second inequality. Equivalently,
        \[
            n-k\geq b^2,\qquad
            n-k\geq\frac{k^{2\ell}}{b^{2\ell-1}}.
        \]
        If $b\geq k^{2\ell/(2\ell+1)}$, the first bound proves \eqref{eq:ordergrowth}; otherwise the second bound proves it.
    \end{proof}

    \begin{Prop}\label{prop:finite}
        Under the hypotheses of Lemma~\ref{lem:moments}, suppose $\sig(G)\geq1/M$, where $M>1$. Then
        \begin{equation}\label{eq:Mconstraints}
            t\leq M(1-x),\qquad
            x^{2\ell-1}(M-1-Mx)\geq1,\qquad
            k\leq\frac{M-1-Mx}{x^2}.
        \end{equation}
        In particular,
        \begin{equation}\label{eq:coarsefinite}
            k<(M-1)^{(2\ell+1)/(2\ell-1)},\qquad
            n<M(M-1)^{(2\ell+1)/(2\ell-1)}.
        \end{equation}
    \end{Prop}
    \begin{proof}
        Since $\sig(G)=(1-x)/t$, the threshold gives the first inequality in~\eqref{eq:Mconstraints}. The other two follow from \eqref{eq:basicbounds}. Also $t<M$, so $x>(M-1)^{-1/(2\ell-1)}$ and $k\leq(t-1)/x^2<(M-1)^{(2\ell+1)/(2\ell-1)}$. Multiplication by $t<M$ gives the order bound.
    \end{proof}

    Thus, for every $\varepsilon>0$, only finitely many triangle-free distance-regular graphs have $\sig(G)\geq\varepsilon$, up to isomorphism. This follows by taking $\ell=1$ and a sufficiently large $M$ with $1/M\leq\varepsilon$, and then bounding the order. In particular, the positive values of $\sig$ in this class have no positive accumulation point. This finiteness statement does not require a diameter bound.

    \begin{Prop}\label{prop:182}
        If $G\in\cl$ and $\sig(G)\geq1/36$, then
        \begin{equation}\label{eq:182}
            k\leq182,\qquad 5n<117k.
        \end{equation}
        If the stronger condition $\sig(G)\geq1/35$ holds, then $k\leq165$ and $5n<112k$.
    \end{Prop}
    \begin{proof}
        Use $\ell=2$ in Proposition~\ref{prop:finite}. For $M=36$,
        \begin{equation}\label{eq:cutineq}
            x^3(35-36x)\geq1,\qquad
            k\leq\frac{35-36x}{x^2}.
        \end{equation}
        The function $x^3(35-36x)$ is increasing on $[0,7/20]$, and
        \[
            \left(\frac7{20}\right)^3
            \left(35-36\frac7{20}\right)
            =\frac{2401}{2500}<1.
        \]
        Hence $x>7/20$. The function $(35-36x)/x^2$ is decreasing on $(0,1)$, so
        \[
            k<\frac{35-36(7/20)}{(7/20)^2}
            =\frac{8960}{49}<183.
        \]
        Also $t\leq36(1-x)<117/5$, proving~\eqref{eq:182}.

        For $M=35$, the same argument uses $x>9/25$, because
        \[
            \left(\frac9{25}\right)^3\left(34-35\frac9{25}\right)
            =\frac{78003}{78125}<1.
        \]
        It gives $k<13375/81<166$ and $t<112/5$.
    \end{proof}

    The stronger cutoff for $1/35$ is used only for the optional consistency check included in the program. The $1/36$ classification alone implies all three extremal conclusions.

    \section{Proof of main results}\label{sec:certificate}
    We now describe a finite calculation proving that only three arrays can meet the threshold. Every rejection uses a necessary condition for an actual graph. No inference of graph existence is made from passing a feasibility test.

    \subsection{Eigenvalues and their multiplicities}
    For $L$ in~\eqref{eq:L}, direct expansion gives
    \begin{align}
        \chi(z)&=\det(zI-L)=(z-k)f(z),\label{eq:chi}\\
        f(z)&=z^3+cz^2+\bigl(\mu(c-k+1)-k\bigr)z-c(k-\mu).
        \label{eq:cubic}
    \end{align}
    The constant term of $f$ is nonzero, so the three nonprincipal eigenvalues are nonzero and Lemma~\ref{lem:rank} applies to each. The $(0,0)$ cofactor of $zI-L$ is
    \begin{equation}\label{eq:R}
        R(z)=(z-k+c)\bigl(z^2-\mu(k-1)\bigr)-cz(k-\mu).
    \end{equation}

    \begin{Lem}\label{lem:residue}
        For each root $\theta$ of $\chi$,
        \begin{equation}\label{eq:residue}
            m(\theta)=\frac{nR(\theta)}{\chi'(\theta)}.
        \end{equation}
        Consequently, for every irreducible factor $g\in\mathbb Q[z]$ of $f$, there is an integer $m_g\geq k$ such that
        \begin{equation}\label{eq:modtest}
            nR(z)\equiv m_g\chi'(z)\pmod{g(z)}.
        \end{equation}
    \end{Lem}
    \begin{proof}
        The symmetric matrix $B$ in Section~\ref{sec:prelim} is the compression of $A$ to the normalized characteristic vectors of the distance layers about a vertex $v$. This invariant subspace contains the coordinate vector at $v$. Therefore the $(0,0)$ entry of its resolvent has the spectral expansion
        \[
            ((zI-B)^{-1})_{00}
            =\sum_\theta\frac{(E_\theta)_{vv}}{z-\theta}
            =\frac1n\sum_\theta\frac{m(\theta)}{z-\theta}.
        \]
        The diagonal entries of each spectral idempotent are constant in a distance-regular graph. Similarity between $B$ and $L$ preserves this resolvent entry, which by the cofactor formula equals $R(z)/\chi(z)$. Taking the residue at a simple root proves~\eqref{eq:residue}.

        The characteristic polynomial of the integer adjacency matrix has integer coefficients. Conjugate algebraic eigenvalues therefore occur with the same integer multiplicity $m_g$. Equation~\eqref{eq:residue}, applied to the roots of $g$, implies~\eqref{eq:modtest}; Lemma~\ref{lem:rank} gives $m_g\geq k$.
    \end{proof}

    The congruence can be checked without computing any algebraic root. Reduce $nR$ and $\chi'$ modulo $g$. Their coefficient vectors must be proportional by a positive integer at least $k$. A nonzero coefficient of the latter vector determines the only possible proportionality constant. The latter vector is not zero, because $\chi$ has simple roots. This is exact polynomial division over $\mathbb Q$.

    Factoring $f$ also requires only elementary arithmetic. A reducible monic integer cubic has an integer root. All roots of $f$ belong to $(-k,k)$, so testing the integers from $-k$ to $k$ finds such a root whenever one exists. The remaining quadratic is reducible over $\mathbb Q$ if and only if its discriminant is a square. This explains every factorization step in the publicly available primary certificate.

    \subsection{Checking the spectral bound}
    For an integer $M>0$, write $h=Mk-n$. The inequality $\sig(G)\geq1/M$ is equivalent to
    \[
        B-\left(\frac nM-k\right)I\succeq0.
    \]
    Let $D_i$ denote the determinant of the leading $i$-by-$i$ principal submatrix of $M(B-(n/M-k)I)$. Tridiagonal determinant recurrence gives
    \begin{align}
        D_1&=h,\nonumber\\
        D_2&=h^2-M^2k,\nonumber\\
        D_3&=h\bigl(D_2-M^2\mu(k-1)\bigr),\label{eq:minors}\\
        D_4&=\bigl(h+M(k-c)\bigr)D_3-M^2c(k-\mu)D_2.\nonumber
    \end{align}
    Since the off-diagonal entries of $B$ are positive, strict interlacing implies that its proper leading principal submatrices have least eigenvalue strictly greater than $\lambda_{\min}(B)$. Equivalently, positive semidefiniteness at the threshold is exactly
    \begin{equation}\label{eq:PSD}
        D_1>0,\quad D_2>0,\quad D_3>0,\quad D_4\geq0.
    \end{equation}
    For sufficiency one may also use the $LDL^{\mathsf T}$ factorization: its first three pivots are positive and its last pivot is nonnegative. The weak inequality in the last test retains equality cases.

    \subsection{Search results}
    By Proposition~\ref{prop:182}, every graph relevant to Theorem~\ref{thm:threshold} has parameters in the finite set
    \[
        2\leq k\leq182,\qquad 1\leq\mu\leq c<k.
    \]
    For each triple, the certificate checks the integrality conditions in~\eqref{eq:layers}, the bound $5n<117k$, the determinant conditions \eqref{eq:PSD} with $M=36$, and the multiplicity conditions \eqref{eq:modtest}. The complete executable program is available from the repository listed in Subsection~\ref{sec:materials}. Its output is summarized in Tables~\ref{tab:counts} and~\ref{tab:survivors}.

    \begin{table}[H]
        \small
        \linespread{1}\selectfont
        \setlength{\abovecaptionskip}{0pt}
        \setlength{\belowcaptionskip}{6pt}
        \centering
        \caption{Exact counts for parameter triples. Each rejection is assigned to the first failed test in the primary program's fixed factor order. The three rejection counts depend on this order; their sum is $1\,662$.}
        \label{tab:counts}
        \begin{tabular}{lr}
            \toprule
            Stage & Number\\
            \midrule
            All triples in the proved range & $1\,004\,731$\\
            Integer distance-layer sizes & $21\,834$\\
            Order bound $5n<117k$ & $5\,778$\\
            Spectral threshold $\sig\geq1/36$ & $1\,665$\\
            Nonconstant conjugate multiplicity: rejected & $1\,531$\\
            Nonintegral multiplicity: rejected & $130$\\
            Multiplicity less than $k$: rejected & $1$\\
            All tests passed & $3$\\
            \bottomrule
        \end{tabular}
    \end{table}

    \begin{table}[H]
        \small
        \linespread{1}\selectfont
        \setlength{\abovecaptionskip}{0pt}
        \setlength{\belowcaptionskip}{6pt}
        \centering
        \caption{All surviving intersection arrays and their spectra. Superscripts indicate multiplicities.}
        \label{tab:survivors}
        \begin{tabular}{lrl}
            \toprule
            $(k,\mu,c)$ & $n$ & Adjacency spectrum\\
            \midrule
            $(2,1,1)$ & $7$ & $\{2^{(1)},(2\cos(2\pi j/7))^{(2)}:1\leq j\leq3\}$\\
            $(4,1,2)$ & $35$ & $\{4^{(1)},2^{(14)},(-1)^{(14)},(-3)^{(6)}\}$\\
            $(7,2,3)$ & $64$ & $\{7^{(1)},3^{(21)},(-1)^{(35)},(-5)^{(7)}\}$\\
            \bottomrule
        \end{tabular}
    \end{table}

    A second program, \texttt{verify\_independent\_sympy.py}, derives the characteristic polynomial directly from the matrix $L$, tests the threshold by exact Sturm root counts, and replaces the residue test with rational trace equations for powers $0,1,2,3,5$. It gives the same $1\,665$ threshold candidates and three survivors: $1\,657$ systems are inconsistent, and five further candidates fail the multiplicity conditions. It also checks agreement between the Sturm and determinant tests for all $5\,778$ triples satisfying the order bound.

    To specify the trace test, let $g$ range over the distinct irreducible factors of $f$ and put $s_j(g)=\sum_{g(\theta)=0}\theta^j$, with $s_0(g)=\deg g$. Newton's identities compute these sums over $\mathbb Q$. There is one unknown $m_g$ for each factor, since conjugate eigenvalues of an integer matrix have equal multiplicity. The equations are
    \[
        \sum_g m_g s_j(g)=\tr(A^j)-k^j,
        \qquad j\in\{0,1,2,3,5\},
    \]
    where the right-hand sides are $n-1,-k,nk-k^2,-k^3,-k^5$, respectively. The coefficient matrix for powers $0,1,2$ has full column rank: it is the invertible Vandermonde matrix on the three distinct roots of $f$ multiplied by the matrix whose columns indicate the disjoint root sets of the factors. Thus a compatible system has a unique solution, and powers $3$ and $5$ impose additional consistency conditions.

    In particular, for $C_7$ the cubic $g(z)=z^3+z^2-2z-1$ is irreducible over $\mathbb Q$, so a single common multiplicity $m$ is required. Its five trace equations are
    \[
        m(3,-1,5,-4,-16)=(6,-2,10,-8,-32).
    \]
    They have the unique solution $m=2$, which satisfies every equation exactly and meets the bound $m\geq k=2$. The same trace procedure therefore covers irreducible cubic factors without a special case.

    The primary certificate uses only the Python standard library; the supplementary check uses SymPy~1.14.0. The second program does not import the first: the two implementations share the proved parameter range but use different algebraic tests. Both programs and their expected outputs are available from the repository listed in Subsection~\ref{sec:materials}.

    \subsection{The three graphs and the proof}\label{sec:proof}
    The three surviving arrays are realized by familiar graphs, whose parameters we recall to make the equality cases explicit.

    The cycle $C_7$ has intersection array $\{2,1,1;1,1,1\}$ and spectrum as in Table~\ref{tab:survivors}. Any connected $2$-regular graph of order seven is $C_7$.

    For $\Odd=KG(7,3)$, fix a $3$-subset $X$. The vertices in the four distance layers have intersection sizes with $X$ equal to $3,0,2,1$, respectively. Their numbers are $1,4,12,18$. Counting disjoint $3$-subsets gives the intersection array $\{4,3,3;1,1,2\}$. For example, if $Y$ meets $X$ in one point, precisely two of the four $3$-subsets of the complement of $Y$ meet $X$ in two points, and two meet $X$ in one point. The cubic in~\eqref{eq:cubic} factors as $(z-2)(z+1)(z+3)$, giving $\sig(\Odd)=1/35$.

    The vertices of $\FQ$ are the antipodal pairs in the binary $7$-cube; two pairs are adjacent when representatives differ in one coordinate. Relative to a fixed pair, distances are the minimum of a Hamming weight and its complement. The layer sizes are $1,7,21,35$, and coordinate flips give the array $\{7,6,5;1,2,3\}$. Characters of the binary cube that are constant on antipodal pairs have even weights $0,2,4,6$. Their eigenvalues are $7,3,-1,-5$, with multiplicities $1,21,35,7$. In particular, $\sig(\FQ)=2/64=1/32$.

    For uniqueness, we use the Main Theorem of Huang and Liu~\cite[p.~196]{HL1999}. It states that a connected regular graph $H$ cospectral with a generalized odd graph $\Gamma$ is distance-regular with the same intersection array as $\Gamma$; moreover, $H\cong\Gamma$ when $\Gamma$ is an odd cycle, an Odd graph, or a folded odd cube, among the families listed there. Any distance-regular realization of either of the last two arrays in Table~\ref{tab:survivors} is connected and regular, and the array determines its full spectrum through~\eqref{eq:residue}. These spectra are exactly those of $\Odd$ and $\FQ$, respectively. The cited theorem therefore gives the required isomorphisms.

    \begin{proof}[Proof of Theorem~\ref{thm:threshold}]
        Suppose $G\in\cl$ and $\sig(G)\geq1/36$. Its array has the form \eqref{eq:array}. Proposition~\ref{prop:182} puts it in the proved finite range. Equations~\eqref{eq:layers}, \eqref{eq:PSD}, and Lemma~\ref{lem:residue} show that it passes every test of the exact certificate. Hence its array is one of the three in Table~\ref{tab:survivors}. The uniqueness facts just stated identify $G$ as $C_7$, $\Odd$, or $\FQ$.

        Conversely, the preceding constructions belong to $\cl$. The values for $\Odd$ and $\FQ$ exceed $1/36$. For $C_7$, let $\alpha$ be the least root of $p(z)=z^3+z^2-2z-1$. Direct rational calculation gives
        \[
            p(-65/36)=-\frac{701}{46656}<0,\qquad
            p(-9/5)=\frac1{125}>0.
        \]
        For $z\leq-9/5$, $p'(z)=3z^2+2z-2\geq103/25>0$, so $p$ is strictly increasing there and $-65/36<\alpha<-9/5$. Consequently
        \[
            \frac1{36}<\frac{2+\alpha}{7}
            =\frac{2(1-\cos(\pi/7))}{7}<\frac1{35}.\qedhere
        \]
    \end{proof}

    \begin{proof}[Proof of Corollary~\ref{cor:extrema}]
        Theorem~\ref{thm:threshold} leaves only the three identified graphs at or above $1/36$. Their values, computed above, satisfy
        \[
            \frac1{36}<\sig(C_7)<\frac1{35}=\sig(\Odd)
            <\frac1{32}=\sig(\FQ).
        \]
        The three extremal conclusions and all equality statements follow.
    \end{proof}

    \section{Further results and discussion}\label{sec:discussion}
    The odd-girth theorem of Lee and Weng~\cite{LW2012}, with the direct proof of van Dam and Fiol~\cite{DF2012}, states that a connected graph with $d+1$ distinct adjacency eigenvalues and finite odd girth at least $2d+1$ is a distance-regular generalized odd graph of diameter $d$. No regularity assumption is needed.

    \begin{Cor}\label{cor:four}
        Let $G$ be a connected $\{C_3,C_5\}$-free graph with at most four distinct adjacency eigenvalues. Then $\sig(G)\leq1/32$, with equality if and only if $G\cong\FQ$. The threshold classification and the second and third extremal statements of Theorem~\ref{thm:threshold} and Corollary~\ref{cor:extrema} remain valid in this class.
    \end{Cor}
    \begin{proof}
        If $G$ is bipartite, then $\sig(G)=0$. Otherwise its odd girth is finite and at least seven. Write $d+1$ for its number of distinct eigenvalues, so $d\leq3$. The odd-girth theorem implies that $G$ is a generalized odd graph with odd girth $2d+1$. Hence $d=3$, and $G\in\cl$. Apply the main theorem.
    \end{proof}

    \begin{Cor}
        For $G\in\cl$, let $q_n(G)$ be the least eigenvalue of $Q(G)=D(G)+A(G)$. Then $q_n(G)\leq n/32$, with equality exactly for $\FQ$. After excluding $\FQ$, the sharp bound is $q_n(G)\leq n/35$, with equality exactly for $\Odd$. The third extremal value is $2(1-\cos(\pi/7))/7$ after division by $n$.
    \end{Cor}
    \begin{proof}
        Distance-regular graphs are regular. Thus $Q=kI+A$ and $q_n(G)=k+\lambda_n(G)$, so the result is precisely Corollary~\ref{cor:extrema}.
    \end{proof}

    These are direct consequences of the main classification and established spectral results.

    \subsection{Related work and open questions}
    Table~\ref{tab:comparison} locates the result relative to the pertinent literature. The comparison distinguishes a sharp result in a structured class from bounds valid for all graphs with the same forbidden cycles.

    \begin{table}[H]
        \small
        \linespread{1}\selectfont
        \setlength{\abovecaptionskip}{0pt}
        \setlength{\belowcaptionskip}{6pt}
        \caption{Direct antecedents, related results, and the contribution of the present classification.}
        \label{tab:comparison}
        \begin{tabularx}{\textwidth}{@{}>{\raggedright\arraybackslash}p{33mm}>{\raggedright\arraybackslash}X>{\raggedright\arraybackslash}X@{}}
            \toprule
            Work & Scope and method & Relation to this paper\\
            \midrule
            Terwilliger (1982)~\cite{Ter1982}; Juri\v{s}i\'{c}--Koolen--Miklavi\v{c} (2005)~\cite{JKM2005} & Tree bounds and neighborhood Gram matrices give the established local multiplicity bound. & This bound supplies the multiplicity constraint used in the explicit valency reduction.\\[4pt]
            Abiad--Taranchuk--van Veluw (2026)~\cite{ATV2026} & Lemma~6 gives the odd-moment inequality; weighted interlacing gives $\sig<0.0396$ for all $\{C_3,C_5\}$-free graphs. & Combining their moment bound with multiplicities yields a finite search answering Problem~11 and determining two further extrema.\\[4pt]
            Yip (2026)~\cite{Yip2026} & High odd girth tending to infinity; Chebyshev-polynomial and approximation methods give $O(g^{-3}(\log g)^3)$. & The present theorem determines exact constants at fixed odd girth seven and diameter three.\\[4pt]
            Qiao--Koolen (2018)~\cite{QK2018} & Valency bounds and classifications under assumptions on the least eigenvalue relative to the valency. & A threshold for $\sig$ gives the explicit cutoff $182$ and a short arithmetic exhaustion.\\[4pt]
            Lang--Terwilliger (2007)~\cite{LT2007} & Almost-bipartite distance-regular graphs with the $Q$-polynomial property. & No $Q$-polynomial assumption or such classification is used.\\[4pt]
            Abiad--Kumar--Pragada (2026 preprint)~\cite{AKP2026} & Forbidden-clique spectral bipartiteness bounds via localized independent set inequalities. & Different constraints; the preprint supplies context, not an exact diameter-three extremal theorem.\\
            \bottomrule
        \end{tabularx}
    \end{table}

    The new conclusion is the exhaustive exclusion of every competing intersection array above $1/36$, including possible arrays outside existing classifications. The multiplicity and moment bounds, the three graph constructions, and their uniqueness are established ingredients. The cutoff and exact exclusion turn them into a complete extremal classification.

    The proof is computer-assisted. Proposition~\ref{prop:182} proves the finite range analytically, and the exact enumeration in Section~\ref{sec:certificate} excludes every other array in that range. The constructions and Huang--Liu theorem then establish existence and uniqueness. The second program cross-checks the finite exclusion; it is not needed for the implication from the primary certificate to Theorem~\ref{thm:threshold}. The versioned repository listed below makes the calculation reproducible.

    The sharp bound for all $\{C_3,C_5\}$-free regular graphs remains open. Uniform independent blow-ups of $\FQ$ preserve $\sig=1/32$ and odd girth seven, but generally leave the distance-regular class. Extending the theorem beyond four distinct eigenvalues requires additional structure; the present result also does not give an edit-distance stability theorem or a sharp edge-deletion bound for bipartiteness.

    The remaining natural questions are whether the finite parameter exclusion admits a substantially shorter symbolic proof, and how far the same threshold method can be pushed at larger diameter. Proposition \ref{prop:finite} supplies a finite reduction for any fixed positive target, but the number of intersection parameters then increases.

    \subsection{Programs and data}\label{sec:materials}
    The complete computational materials are publicly available at \url{https://github.com/zhouqi-math/generalized-odd-graph-spectral-extrema}. The repository contains the primary exact certificate, the independent Sturm-and-trace verification, the graph-construction checks, execution instructions, and all expected JSON reports and candidate records. All computational outputs were independently verified by the author following the validation procedures described herein.

%%%%%%%%%%%%%%%%%%%%%%%%%%%%%%%%%%%%%%%%%%%%%
\subsection*{Acknowledgements}
%%%%%%%%%%%%%%%%%%%%%%%%%%%%%%%%%%%%%%%%%%%%%
The author is supported by the National Natural Science Foundation of China (Nos. 12401442).

\subsection*{Declaration of AI use}

The author acknowledges the use of ChatGPT (GPT-5.6, OpenAI) as an auxiliary tool for literature retrieval, language refinement, and programming.  AI tools were not used to draft the manuscript. The author retains full responsibility for the content, proofs, and conclusions of the final manuscript.

\end{document}